\documentclass[11pt, oneside]{article}   	
\usepackage{geometry}                		
\usepackage{graphicx}				
								
\usepackage[all]{xy}
\CompileMatrices

\usepackage{amssymb}
\usepackage{amsmath}
\usepackage{stmaryrd}
\usepackage{amsthm}
\usepackage{tikz-cd}
\usepackage{extarrows}
\usepackage[mathscr]{euscript}
\usepackage{mathrsfs} 
\usepackage{verbatim}
\usepackage[OT2,T1]{fontenc}

\usepackage[noinputenc,nocaptions]{vietnam}
\catcode`\@=11
\def\tabb{\@tabacckludge}
\catcode`\@=12

\usepackage{rotating}

\usepackage{bm}

\DeclareSymbolFont{cyrletters}{OT2}{wncyr}{m}{n}
\DeclareMathSymbol{\Sha}{\mathalpha}{cyrletters}{"58}
\DeclareMathSymbol{\Che}{\mathalpha}{cyrletters}{"51}

\usepackage{calligra}
\usepackage{mathrsfs}

\newcommand{\calMor}{\mathscr{M}\mathit{or}}

\newcommand{\Ga}{{\mathbf{G}}_{\rm{a}}}
\newcommand{\Gm}{{\mathbf{G}}_{\rm{m}}}

\DeclareMathOperator{\Pic}{Pic}

\DeclareMathOperator{\ab}{\rm{ab}}
\DeclareMathOperator{\Ext}{Ext}

\DeclareMathOperator{\Hom}{Hom}

\DeclareMathOperator{\R}{R}

\DeclareMathOperator{\uni}{uni}

\DeclareMathOperator{\coker}{coker}
\DeclareMathOperator{\perf}{perf}
\DeclareMathOperator{\im}{im}

\newcommand*{\F}{\ensuremath{\mathbf{F}}}                        
\newcommand*{\A}{\ensuremath{\mathbf{A}}}                        
\renewcommand*{\P}{\ensuremath{\mathbf{P}}}                        

\numberwithin{equation}{section}

\newtheorem{theorem}{Theorem}[section]

\newtheorem{lemma}[theorem]{Lemma}
\newtheorem{proposition}[theorem]{Proposition}
\newtheorem{corollary}[theorem]{Corollary}

\theoremstyle{definition}
  \newtheorem{definition}[theorem]{Definition}

\theoremstyle{remark}
  \newtheorem{remark}[theorem]{Remark}

\theoremstyle{definition}
  \newtheorem{example}[theorem]{Example}

\theoremstyle{remark}

\usepackage[OT2,T1]{fontenc}

\tikzset{commutative diagrams/.cd,
mysymbol/.style={start anchor=center,end anchor=center,draw=none}
}

\usepackage{stmaryrd}
\usepackage{hyperref}

\title{\textbf{EXTENSIONS OF UNIRATIONAL GROUPS}}
\author{Zev Rosengarten \thanks{
While completing this work, the author was supported by Israel Science Foundation Grant No.\,2083/24.
}}

\date{}
\begin{document}
\maketitle

\begin{abstract}
We undertake a study of extensions of unirational algebraic groups. We prove that extensions of unirational groups are also unirational over fields of degree of imperfection $1$, but that this fails over every field of higher degree of imperfection, answering a question of Achet. We also initiate a study of those groups which admit filtrations with unirational graded pieces, and show that one may deduce unirationality of unipotent groups from unirationality of certain quotients.
\end{abstract}

\setcounter{tocdepth}{1}
\tableofcontents{}

\section{Introduction}

Recall that a finite type $K$-scheme $X$ is said to be {\em unirational} when there is a dominant rational map $\A^n \rightarrow X$ for some $n > 0$. By dominant, we mean in the sense that the schematic closure is all of $X$, so in particular $X$ must be geometrically integral. Unirational varieties are of interest for a variety of reasons, one of which is that they admit many (a Zariski dense set of) rational points, at least when $K$ is infinite. Unirational algebraic groups, in particular, have received attention from various authors; see, for instance, \cite[\S10.3]{neronmodels} (on the relationship between unirationality and the existence of N\'eron models for algebraic groups), \cite[\S11.3]{cgp} (largely on the relationship between unirationality and various arithmetic properties of algebraic groups), \cite{achetunir} (regarding geometric properties of unirational groups), \cite[\S6]{rostrans} (about the Picard groups of unirational groups), and \cite{rosrigidity} (concerned primarily with various rigidity properties exhibited by unirational groups, in some sense analogous to the well-known rigidity properties of abelian varieties). Over perfect fields, every connected linear algebraic groups is unirational \cite[Ch.\,V,Th.\,18.2(ii)]{borel}, but this is false over every imperfect field \cite[Ex.\,11.3.1]{cgp}, and so unirationality becomes a nontrivial condition on (connected) linear algebraic groups in the imperfect setting. If $G$ is a connected linear algebraic group, and $T \subset G$ is a maximal $K$-torus, then $G$ is unirational precisely when the unipotent group $Z_G(T)/T$ is, where $Z_G(T)$ denotes the centralizer of $T$. (Combine \cite[Prop.\,7.12]{rosrigidity} and Lemma \ref{extunitor} below.) For this reason (and others), unirationality questions for algebraic groups often reduce to the unipotent setting, and thus the theory of unipotent groups plays a central role in such problems.

In \cite[Question 4.9]{achetpicunip}, Achet posed the following question: Is a commutative extension of two unirational algebraic $K$-groups necessarily also unirational? His motivation lies in the study of the maximal unirational $K$-subgroup $G_{\mathrm{uni}}$ of an algebraic group. (This subgroup of $G$ may also be described as the group generated by all maps from unirational varieties into $G$ which pass through the identity.) Achet posed his question because an affirmative answer implies that, for commutative $G$, the group $G/G_{\mathrm{uni}}$ admits no nontrivial unirational subgroup. The question trivially admits an affirmative answer when $K$ is perfect, since all connected linear algebraic groups are unirational in this case, but over imperfect fields it becomes highly nontrivial.

In this paper, we answer Achet's question, and more generally undertake a study of extensions of unirational algebraic groups. In \S\ref{degimp1sec}, we prove that every extension of unirational groups is unirational over fields of degree of imperfection $1$ (Theorem \ref{extunirdegimp1}). This makes crucial use of the concept of permawound unipotent groups introduced in \cite{rospermawound}. We will remind the reader of the definition of permawoundness in \S\ref{degimp1sec}, but their utility arises not from their definition, but from the fact that they simultaneously exhibit two important properties: One the one hand, they enjoy a certain universality or ubiquity -- every commutative $p$-torsion wound unipotent group may be embedded into a permawound one -- so that many questions about arbitrary unipotent groups reduce to the permawound case. On the other hand, they exhibit an extraordinary rigidity: After passage to a suitable finite separable extension, every permawound group admits a filtration with successive quotients all isomorphic to one of two particular groups, very much analogously to the manner in which split unipotent groups always (by definition) admit a filtration with $\Ga$ quotients \cite[Th.\,1.4,1.5]{rospermawound}. One sign that they are relevant for our investigations here is that such groups are always unirational \cite[Th.\,1.9(i)]{rosrigidity}. In fact, we shall show that a smooth unipotent group over a field of degree of imperfection $1$ is unirational precisely when it is permawound (Proposition \ref{perm=unirdegimp1}), from which the stability of unirationality under extensions will follow easily.

However, this stability fails over every field of degree of imperfection $> 1$. In fact, in \S\ref{examplessec} we construct over every such field a commutative $p$-torsion extension of unirational wound unipotent groups which is not unirational (Example \ref{extwdbywdnotunirex}). We also construct a non-unirational commutative $p$-torsion wound unipotent extension of the additive group by a unirational wound unipotent group (Example \ref{unirnotinherexts}). Thus Achet's question has a negative answer over every field of degree of imperfection $>1$.

This naturally leads one to introduce a new class of algebraic groups: those which admit a filtration whose successive quotients are unirational. We call such groups {\em ext-unirational}, and in \S\ref{extunirsec} we show that ext-unirationality of a connected linear algebraic $G$ over a field of characteristic $p$ is equivalent to ext-unirationality of the maximal commutative $p$-torsion quotient $G^{\ab}/[p]G^{\ab}$ (Proposition \ref{extunirunipptorcom}), as well as to that of the centralizer $Z_G(T)$ of any torus $T$ of $G$ (Proposition \ref{Gextunir}).

Finally, in \S\ref{dedfromquotssec}, we continue with the theme of deducing unirationality properties from quotients, but for unirationality itself rather than ext-unirationality. More precisely, we prove that over a field $K$ of degree of imperfection $r > 1$, a smooth connected unipotent $K$-group $U$ is unirational if and only if the quotient $U/\mathscr{D}_rU$ of $U$ by the $r$th term in its lower central series is (Theorem \ref{unifrmquot}). Of particular note is the case $r \leq 2$, which -- in conjunction with Corollary \ref{uniriffabis} and the unirationality of all connected linear algebraic groups when $r = 0$ -- tells us that, when $K$ has degree of imperfection $\leq 2$, unirationality of $U$ is equivalent to that of its maximal commutative quotient $U^{\ab}$. We show that this is optimal by giving, over every field of degree of imperfection $\geq 3$, an example of a wound unipotent $K$-group $U$ such that $U^{\ab}$ is $p$-torsion and unirational even though $U$ fails to be unirational. (Example \ref{Uabunirex}).

\subsection*{Notation and Conventions} Throughout this paper, $K$ denotes a field, and when it appears, $p$ denotes a prime number equal to the characteristic of $K$. By a {\em linear algebraic $K$-group}, or a {\em linear algebraic group over $K$}, we shall mean a smooth affine $K$-group scheme. Recall that the {\em degree of imperfection} of a field $K$ of characteristic $p$ is defined to be $\mathrm{log}_p[K: K^p]$. This quantity is always either infinite or a nonnegative integer. The degree of imperfection of $K$ also equals $\mathrm{dim}_K(\Omega^1_{K/\F_p})$, as well as (tautologically, from the first definition) the size of any $p$-basis of $K$ \cite[Th.\,26.5]{matsumura}.

\subsection*{Acknowledgments}

I thank Anis Zidani for pointing out an error in an earlier version of Theorem \ref{unifrmquot}.

\section{Degree of imperfection $1$}
\label{degimp1sec}

In this section, we will prove that an arbitrary (not necessarily commutative) extension of unirational algebraic $K$-groups must be unirational when $K$ has degree of imperfection $1$. We begin by tying up a loose end concerning a basic permanence property of permawound unipotent groups, which were introduced in \cite{rospermawound}. Recall \cite[Def.\,1.2]{rospermawound} that a smooth unipotent $K$-group scheme $U$ is called {\em permawound} when, for any exact sequence of finite type $K$-group schemes $$U \longrightarrow E \longrightarrow \Ga \longrightarrow 1,$$ $E$ contains a copy of $\Ga$; that is, there is an injective (or equivalently, there is a nonzero) $K$-homomorphism $\Ga \rightarrow E$. As we shall show below (Proposition \ref{perm=unirdegimp1}), it is an easy consequence of the results of \cite{rospermawound} and \cite{rosrigidity} that permawoundness is equivalent to unirationality for unipotent groups over fields of degree of imperfection $1$. At least for unipotent groups, then, we will deduce the permanence of unirationality under extensions from the corresponding property of permawoundness, to be proved presently. The case of general algebraic groups will then be obtained by reduction to the unipotent case, using the fact, a consequence of \cite[Prop.\,7.12]{rosrigidity} and Lemma \ref{extunitor} below, that $G$ is unirational if and only if the same holds for $Z_G(T)/T$, where $T \subset G$ is a maximal torus.

\begin{proposition}
\label{permawdexts}
Given an exact sequence $$1 \longrightarrow U' \longrightarrow U \longrightarrow U'' \longrightarrow 1$$ of unipotent $K$-groups with $U', U''$ permawound, $U$ is also permawound.
\end{proposition}

\begin{proof}
If $K$ is perfect, then the assertion is trivially true by \cite[Prop.\,5.2(i)]{rospermawound}, so assume that $K$ is imperfect. Then $U'$ and $U''$ are connected \cite[Prop.\,6.2]{rospermawound}, hence so is $U$. We proceed by dimension induction. Note first that, since every quotient of $U$ is an extension of a quotient of $U''$ by a quotient of $U'$, and since permawoundness is (by definition) inherited by quotients, it follows that we may assume that every quotient of $U$ by a positive-dimensional normal $K$-subgroup is permawound. Consider first the case in which $U$ is not wound. Suppose that we have an exact sequence of finite type $K$-groups $$U \xlongrightarrow{f} E \longrightarrow \Ga \longrightarrow 1.$$ We must show that $E$ contains a copy of $\Ga$. If $f$ does not kill the maximal split $K$-subgroup $U_s$ of $U$, then this is immediate. Otherwise, $f(U)$ is a quotient of $U/U_s$, hence permawound, so $E$ contains a copy of $\Ga$ in this case as well. Hence $U$ is permawound.

Now suppose that $U$ is wound. If $U'$ is trivial, the assertion is immediate, so assume that $U' \neq 1$. Then we claim that there is a nontrivial smooth connected $K$-subgroup $V \subset U'$ that is central in $U$. Indeed, define a sequence of subgroups $V_n \subset U'$ by the following recursive formula: $V_0 := U'$, and $V_{n+1} := [U, V_n]$ for $n \geq 0$. Then one verifies by induction that the $V_n$ are normal subgroups of $U$, and that they form a descending sequence. Because $U$ is unipotent, it is nilpotent, hence $V_{n+1} = 0$ for some $n \geq 0$. If we choose $n$ to be minimal, then $V := V_n \subset U'$ is nontrivial and central in $U$, which proves the claim.

Let $\overline{U} := U/V$, let $\overline{U}_s \subset \overline{U}$ denote the maximal split subgroup, and let $\overline{U}_w := \overline{U}/\overline{U}_s$ denote the wound quotient, which is permawound because $V$ is nontrivial. Because $V \subset U$ is central, the commutator map $U \times U \rightarrow U$ descends to a map $c\colon \overline{U} \times \overline{U} \rightarrow U$. We claim that $c$ further descends to a map $\overline{U}_w \times \overline{U}_w \rightarrow U$. Indeed, let $u_1, u_2 \in \overline{U}(K_s)$. Then $c$ restricts to a map $u_1\overline{U}_s \times u_2\overline{U}_s \rightarrow U_{K_s}$ which must be constant because $\overline{U}_s^2$ is split while $U_{K_s}$ is wound. Because $\overline{U}^2(K_s)$ is Zariski dense in $\overline{U}^2$, it follows that $c$ is $\overline{U}_s$-invariant in both arguments, hence $c$ indeed descends to a map $\overline{U}_w \times \overline{U}_w \rightarrow U$ which we by abuse of notation also denote by $c$. Because $\overline{U}_w$ is wound and permawound, it is commutative \cite[Cor.\,10.4]{rosrigidity}, so $c$ is the constant map to $1 \in U(K)$ by \cite[Lem.\,10.2]{rosrigidity}. It follows that $U$ is commutative. Because commutative permawound groups are trivially weakly permawound \cite[Def.\,5.1]{rospermawound}, $U'$ and $U''$ are weakly permawound, hence so is $U$ \cite[Prop.\,5.6]{rospermawound}. It follows that $U$ is permawound \cite[Prop.\,6.9]{rospermawound}.
\end{proof}

Next we verify that, for fields of degree of imperfection $1$, unirationality for unipotent groups is the same as permawoundness.

\begin{proposition}
\label{perm=unirdegimp1}
Let $K$ be a field of degree of imperfection $1$. A unipotent $K$-group $U$ is unirational if and only if it is permawound.
\end{proposition}

\begin{proof}
For the if direction, use \cite[Th.\,1.9(i)]{rosrigidity}. For the converse, \cite[Prop.\,9.6]{rosrigidity} yields the result when $U$ is wound. For arbitrary $U$, unirationality implies that $U$ is smooth and connected. Then we have an exact sequence $$1 \longrightarrow U_s \longrightarrow U \longrightarrow U_w \longrightarrow 1$$ with $U_s$ split and $U_w$ wound. Because $U$ is unirational, so is $U_w$, which is therefore permawound. Then \cite[Prop.\,5.3]{rospermawound} and Proposition \ref{permawdexts} show that $U$ is permawound.
\end{proof}

We require the following simple lemma.

\begin{lemma}
\label{extunitor}
Given an extension of finite type $K$-group schemes $$1 \longrightarrow T \longrightarrow G \longrightarrow U \longrightarrow 1$$ with $T$ a torus and $U$ a unirational $K$-group scheme, $G$ is also unirational.
\end{lemma}

\begin{proof}
We may assume that $K$ is separably closed \cite[Th.\,1.6]{rosrigidity}. Let $f\colon X \rightarrow U$ be a dominant map from a dense open subscheme $X \subset \A^n$. Then $G$ is in particular a $T$-torsor over $U$, and $f^*(G)$ is a $T$-torsor over $X$, which is trivial because $T$ is split and $\Pic(X) = 0$. Thus $f^*(G)$ is rational, so the dominant map $f^*(G) \rightarrow G$ proves that $G$ is unirational.
\end{proof}

We come now to the main result of this section.

\begin{theorem}
\label{extunirdegimp1}
Let $K$ be a field of degree of imperfection $1$, and suppose given an extension $$1 \longrightarrow G' \longrightarrow G \xlongrightarrow{\pi} G'' \longrightarrow 1$$ of finite type $K$-group schemes. If $G'$ and $G''$ are unirational, then so is $G$.
\end{theorem}

\begin{proof}
First, unirational $K$-schemes are necessarily generically smooth and geometrically connected, so $G'$, $G''$ are smooth and connected. They are also linear algebraic, for if not, then -- at least over $\overline{K}$ -- they would admit nonzero abelian variety quotients, which would then be unirational. But a nonzero abelian variety is never unirational. (Any map from a nonempty open subscheme of $\P^1$ into an abelian variety $A$ extends over $\P^1$, and abelian varieties do not admit nonconstant maps from $\P^1$.) It follows that $G$ is also a connected linear algebraic group.

Choose a maximal $K$-torus $T \subset G$, and let $Z_G(T)$ denote the centralizer in $G$ of $T$ (and similarly with the other centralizers below). Let $T'' := \pi(T)$, a maximal $K$-torus of $G''$ \cite[Ch.\,IV, Prop.\,11.14]{borel}, and let $T' := T \cap G'$, a maximal $K$-torus of $G'$ \cite[Cor.\,A.2.7]{cgp}. Then we have an exact sequence 
\begin{equation}
\label{extunirdegimp1pfeqn1}
1 \longrightarrow Z_{G'}(T')/T' \longrightarrow Z_G(T)/T \longrightarrow Z_{G''}(T'')/T'' \longrightarrow 1,
\end{equation}
with exactness on the right using \cite[Ch.\,IV,11.14,Cor.\,2]{borel}. The groups $Z_{G''}(T'')$ and $Z_{G'}(T')$ are unirational \cite[Prop\,7.12]{rosrigidity}, hence so are the two groups on the end in (\ref{extunirdegimp1pfeqn1}). They are also unipotent, the unipotence a consequence of the maximality of the respective tori combined with \cite[Ch.\,IV,11.5,Cor.\,2]{borel}. Now an extension of unirational unipotent groups over a field of degree of imperfection $1$ is still unirational by combining Propositions \ref{perm=unirdegimp1} and \ref{permawdexts}. Thus $Z_G(T)/T$ is unirational, hence so is $Z_G(T)$ (Lemma \ref{extunitor}), hence so is $G$ \cite[Prop.\,7.12]{rosrigidity}.
\end{proof}

\section{Examples of extensions of unirational groups that are not unirational}
\label{examplessec}

In this section we will show by example -- over every field of degree of imperfection $>1$ -- that unirationality is not inherited by extensions in general. Our examples will be of two types: Our first example will be an extension of a wound unirational group by a wound unirational group, while our second will be an extension of $\Ga$ by a wound unirational group. Of course, any extension of a unirational group by a split unipotent one is unirational, as ${\rm{H}}^1(X, \Ga) = 0$ for every affine scheme $X$. Our examples will all be commutative and $p$-torsion, hence they provide a negative answer to Achet's question over fields of degree of imperfection $>1$. The proofs that these examples work will require various auxiliary lemmas, to which we now turn.

For $\alpha \in K$ and $n \geq 1$, we define the $K$-group
\begin{equation}
\label{eqnforVnalpha}
V_{n,\alpha} := \left\{-S + \alpha^{p-1}S^p + \sum_{\substack{0 \leq j < p^n \\ j \not\equiv -1 \pmod{p}}} \alpha^jS_j^{p^n} = 0\right\} \subset \Ga^{p^n-p^{n-1}+1}.
\end{equation}
The next lemma gives a different description of this group. Before stating it, let us recall some basic notions regarding $p$-polynomials. An element of $K[X_1, \dots, X_n]$ is called a {\em $p$-polynomial} when it is a sum of terms of the form $cX_i^{p^r}$ with $c \in K$ and $r \geq 0$. The sum $P$ of the leading terms in each variable $X_i$ is called the {\em principal part} of $F$, and we say that $F$ is {\em reduced} when $P$ has no nontrivial zeroes over $K$: If $P(x_1, \dots, x_n) = 0$ with $x_i \in K$, then $x_i = 0$ for all $i$. We say that $F$ is {\em universal} when the homomorphism that it induces $K^n \rightarrow K$ is surjective. 

\begin{lemma}
\label{descV_nalpha}
If $\alpha \in K-K^p$, then $$V_{n,\alpha} \simeq \R_{K(\alpha^{1/p^{n-1}})/K}\left[\R_{K(\alpha^{1/p^n})/{K(\alpha^{1/p^{n-1}})}}(\Gm)/\Gm\right],$$ where we regard $\Gm$ as sitting inside $\R_{K(\alpha^{1/p^n})/{K(\alpha^{1/p^{n-1}})}}(\Gm)$ in the natural manner, via the inclusion $A^{\times} \subset (A \otimes_{K(\alpha^{1/p^{n-1}})} K(\alpha^{1/p^n}))^{\times}$ for every $K(\alpha^{1/p^{n-1}})$-algebra $A$.
\end{lemma}

Note in particular that $V_{n,\alpha}$ is unirational. This also may be seen as folllows: $V_{n,\alpha}$ is defined over the field $\F_p(\alpha)$, so we may verify its unirationality over this field, which has degree of imperfection $1$. But over $\F_p(\alpha)$, $V_{n,\alpha}$ is permawound \cite[Props.\,6.4,6.9,3.5]{rospermawound}, hence unirational \cite[Th.\,1.9(i)]{rosrigidity}.

\begin{proof}
We first note that all constructions above are valid over $\F_p(\alpha) \subset K$, and commute with arbitrary separable extension on $K$. Thus we may assume that $K = \F_p(\alpha)$, and in particular that $K$ has degree of imperfection $1$. For ease of notation, let us denote $V_{n,\alpha}$ by $V$, and $K(\alpha^{1/p^n}) = K^{1/p^n}$ by $K_n$. We claim that one has an isomorphism of $K_{n-1}$-groups 
\begin{equation}
\label{descV_nalphapfeqn1}
V_{K_{n-1}} \simeq \Ga^{(p-1)(p^{n-1}-1)} \times W, 
\end{equation}
where
\begin{equation}
\label{descV_nalphapfeqn13}
W \simeq \left\{-Y + \alpha^{p-1}Y^p + \sum_{0 \leq j < p-1} \alpha^jY_j^{p^n} = 0\right\} \subset \Ga^p.
\end{equation}
Indeed, over $K_{n-1}$ one may apply the invertible change of variables which fixes $S$ and $S_j$ for $i > p-1$, and such that $$S_j \mapsto S_j - \sum_{\substack{j < i < p^n\\i \equiv j \pmod{p}}} \alpha^{(i-j)/p^n}S_i, \hspace{.3 in} 0 \leq j < p-1.$$ This has the effect of replacing the equation for $V$ to that of $W$ (after renaming the $S$ variables as $Y$ variables), and one has a copy of $\Ga$ for each of the additional variables $S_j$, $j > p-1$, not appearing in the equation, which proves (\ref{descV_nalphapfeqn1}).

We further massage the equation for $W$ by making the invertible change of variables over $K_{n-1}$ which fixes $Y_j$ and does the following to $Y$:
\[
Y \mapsto \alpha^{\frac{1-p^{n-1}}{p^{n-1}}}Y - \sum_{\substack{0 \leq j < p-1\\1 \leq i < n}} \alpha^{\frac{j+1-p^{n-i}}{p^{n-i}}}Y_j^{p^i}.
\]
One verifies that this has the effect (after multiplying through by $\alpha^{\frac{p^{n-1}-1}{p^{n-1}}}$) of replacing the equation (\ref{descV_nalphapfeqn13}) for $W$ with the following equation:
\begin{equation}
\label{descV_nalphapfeqn12}
W \simeq \left\{-Y + \alpha^{\frac{p-1}{p^{n-1}}}Y^p + \sum_{0 \leq j < p-1}\alpha^{\frac{j}{p^{n-1}}}Y_j^p = 0\right\}.
\end{equation}
By \cite[Ch.\,VI,Prop.\,5.3]{oesterle}, this last equation describes the $K_{n-1}$-group $\R_{K_n/K_{n-1}}(\Gm)/\Gm$, so via (\ref{descV_nalphapfeqn1}), we obtain a nonzero $K_{n-1}$-homomorphism $V_{K_{n-1}} \rightarrow \R_{K_n/K_{n-1}}(\Gm)/\Gm$, whence a nonzero homomorphism $g\colon V \rightarrow \R_{K_{n-1}/K}(\R_{K_n/K_{n-1}}(\Gm)/\Gm)$  {\em over $K$}. It remains to prove that this (and indeed, any such) homomorphism is an isomorphism. The principal part of the polynomial (\ref{eqnforVnalpha}) defining $V$ is reduced and universal over $F$ (for instance, using \cite[Prop.\,3.5(ii)]{rospermawound}), so $V$ is permawound over $F$ \cite[Props.\,6.4,6.9]{rospermawound}, hence unirational \cite[Th.\,1.9(i)]{rosrigidity}. Thus $\im(g) \subset \R_{K_{n-1}/K}(\R_{K_n/K_{n-1}}(\Gm)/\Gm)$ is a nonzero unirational $K$-subgroup. But the latter group is a minimal unirational $K$-group in the sense that its only nonzero unirational $K$-subgroup is itself \cite[Props.\,2.18,2.20]{achetunir}, hence $g$ is surjective. To prove that $g$ is injective, we note that dimension considerations imply that $\ker(g)$ is finite. Because $V/\ker(g) \simeq \R_{K_{n-1}/K}(\R_{K_n/K_{n-1}}(\Gm)/\Gm)$ is wound unipotent, $\ker(g)$ is weakly permawound \cite[Prop.\,5.5]{rospermawound}, hence has strictly positive dimension if nontrivial (for instance, by \cite[Th.\,9.5]{rospermawound}; recall that weak permawoundness is insensitive to algebraic separable extension of the ground field \cite[Prop.\,6.7]{rospermawound}). Thus $\ker(g) = 0$, so $g$ is an isomorphism.
\end{proof}

\begin{lemma}
\label{unirimpcertmaps}
Suppose given a separably closed field $K$ and a commutative $p$-torsion unirational wound unipotent $K$-group scheme $U$. Then $U$ is generated by maps of the form $V_{n,\alpha} \rightarrow U$ with $n \geq 1$ and $\alpha \in K - K^p$.
\end{lemma}

\begin{proof}
Because $U$ is unirational, it is generated by maps from nonempty open subschemes of $\P^1$ with a rational point lying above the identity. But any such map $X \rightarrow U$ factors through a homomorphism $\R_{D/K}(\Gm) \rightarrow U$ for some finite closed subscheme $D \subset \P^1$ by \cite[\S 10.3, Thm.\,2]{neronmodels}. (For a more detailed explanation of how this follows from that theorem, see the discussion before Theorem 2.2 in \cite[\S2]{rosmodulispaces}.) Because $U$ is wound, $D$ may be taken to be reduced by \cite[Lem.\,2.3]{rosmodulispaces}. Write $D = \coprod \{x_i\}$ with the $x_i$ (reduced) closed points in $\P^1$. Then $\R_{D/K}(\Gm) \simeq \prod_i \R_{K(x_i)/K}(\Gm)$. Hence $U$ is generated by maps from $\R_{K(x)/K}(\Gm)$ as $x$ varies over closed points of $\P^1$. Because $K$ is separably closed, $K(x)/K$ is a (finite) purely inseparable extension, and it is a primitive extension (generated by a single element) by \cite[Lem.\,5.4]{rosrigidity}. We therefore conclude that $U$ is generated by maps from $\R_{L/K}(\Gm)$ as $L$ varies over primitive nontrivial purely inseparable extensions of $K$. But any such extension is of the form $K(\alpha^{1/p^n})$ for some $\alpha \in K-K^p$ and $n \geq 1$. The multiplication by $p$ map on $\R_{K(\alpha^{1/p^n})/K}(\Gm)$ has image $\R_{K(\alpha^{1/p^{n-1}})/K}(\Gm)$, as may, for instance, be checked on $\F_p(\alpha)_s$-points, since both groups are smooth and defined over $\F_p(\alpha)$. Thus, because $U$ is $p$-torsion, it is generated by maps from groups of the form $$\R_{K(\alpha^{1/p^n})/K}(\Gm)/\R_{K(\alpha^{1/p^{n-1}})/K}(\Gm) \simeq \R_{K(\alpha^{1/p^{n-1}})/K}\left[\R_{K(\alpha^{1/p^n})/{K(\alpha^{1/p^{n-1}})}}(\Gm)/\Gm\right],$$ the last equality by \cite[Cor.\,A.5.4(3)]{cgp}. Now apply Lemma \ref{descV_nalpha}.
\end{proof}

\begin{lemma}
\label{mapVnalphtoVaalph}
Let $\alpha, \lambda \in K - K^p$, and let $f\colon V_{n,\alpha} \rightarrow V_{1,\lambda} \subset \Ga^p$ be a nonzero $K$-group homomorphism. Then $K^p(\alpha) = K^p(\lambda)$, and $f$ is $($in each coordinate$)$ homogeneous linear in $S$, and homogeneous of degree $p^{n-1}$ in the $S_j$. Further, using the variables $X, X_i$ on $V_{1,\lambda}$ $($instead of $S,S_j$$)$, the $X$ coordinate is of the form $aS$ with $0 \neq a \in K$.
\end{lemma}

\begin{proof}
Because $V_{1,\lambda}$ is a minimal unirational $K$-group -- in the sense that its only nonzero unirational $K$-subgroup is itself \cite[Props.\,2.18,2.20]{achetunir} -- $f$ is surjective. The group $V_{n,\alpha}$ splits over $K(\alpha^{1/p^{\infty}})$, hence so does $V_{1,\lambda}$, so we deduce that $\lambda^{1/p} \in K(\alpha^{1/p^\infty})$. It follows that $K^p(\alpha) = K^p(\lambda)$. Let $\{\alpha\} \coprod \mathscr{B}$ be a $p$-basis for $K$. Then the extension $L := K(\mathscr{B}^{p^{-\infty}})$ obtained by adjoining all $p$-power roots of elements of $\mathscr{B}$ to $K$ has $\{\alpha\}$ as a $p$-basis. Further, one has $\alpha \in L^p(\lambda)$, so $\lambda \notin L^p$ and all hypotheses are preserved upon scalar extension to $L$. We may therefore assume that $K$ has degree of imperfection $1$.

We claim that we have for $m \geq 1$ a surjection $f_m\colon V_{m+1,\alpha} \rightarrow V_{m,\alpha}$ with kernel a power of $\R_{K^{1/p}/K}(\alpha_p)$. Indeed, using the equation (\ref{eqnforVnalpha}), if we use the letter $S$ to denote the variables for $V_{m+1,\alpha}$, and the letter $T$ to denote those for $V_{m,\alpha}$, we have the map $(S, (S_j)_j) \mapsto (Z, (Z_j)_j)$, where $Z := S$ and $Z_j := \sum_{i=0}^{p-1} \alpha^iS_{j+p^mi}^p$. The kernel of this map is a power of the $K$-group $$W := \left\{\sum_{i=0}^{p-1}\alpha^iX_i^p = 0\right\} \subset \Ga^p,$$ which is isomorphic to $\R_{K^{1/p}/K}(\alpha_p)$ \cite[Prop.\,7.4]{rospermawound}.

By \cite[Lem.\,7.9]{rospermawound}, therefore, and an easy induction, any map $V_{n,\alpha} \rightarrow V_{1,\lambda}$ must factor through the map $V_{n,\alpha} \rightarrow V_{1,\alpha}$ obtained by iterating the maps $f_m$ above. The lemma therefore reduces to the case $n = 1$, where we must show that the map in question is obtained by a linear change of variables, and that the coefficient of $S$ in $X$ is nonzero. (That $X$ may be written {\em uniquely} as a $p$-polynomial of degree $\leq 1$ in $S$ follows from \cite[Prop.\,6.4]{rosmodulispaces}.) For this, we may extend scalars to $K_s$ and thereby assume that $K$ is separably closed. The lemma then follows from \cite[Lem.\,7.1,Prop.\,9.7]{rospermawound}, since, for a linear change of variables that transforms the polynomial $-X + \lambda^{p-1}X + \sum_{i=0}^{p-2} \lambda^iS_i^p$ into a nonzero scalar multiple of $-S + \alpha^{p-1}S + \sum_{i=0}^{p-2} \alpha^iS_i^p$, $X = X(S, S_i)$ must be a nonzero multiple of $S$.
\end{proof}

Now we give our examples.

\begin{example}
\label{extwdbywdnotunirex}
Let $K$ be a field of degree of imperfection $>1$, and let $\lambda, \mu \in K$ be $p$-independent. (This has various equivalent definitions. For our purposes, the reader can take it to mean that $[K^p(\lambda, \mu): K^p] = p^2$.) We use the equation (\ref{eqnforVnalpha}) for the variables on $V_{1,\lambda}$, using the letter $Y$ instead of $S$ to denote the variables: $Y, Y_0, \dots, Y_{p-2}$. We have an exact sequence 
\begin{equation}
\label{extdefiningv1mu}
0 \longrightarrow V_{1,\mu} \longrightarrow \Ga^p \xlongrightarrow{F_{\mu}} \Ga \longrightarrow 0,
\end{equation}
where $$F_{\mu}(X_0, \dots, X_{p-1}) := -X_{p-1} + \sum_{i=0}^{p-1}\mu^iX_i^p.$$ Consider the homomorphism $f\colon V_{1,\lambda} \rightarrow \Ga$ defined by $Y$, and let $E \in \Ext^1(V_{1,\lambda}, V_{1,\mu})$ be its image under the connecting map $\delta$ associated to (\ref{extdefiningv1mu}). Then we claim that $E$ is not unirational, despite the fact that $V_{1,\lambda}$ and $V_{1,\mu}$ are.

In proving that $E$ is not unirational, we may assume that $K$ is separably closed. Suppose that $E$ is unirational, and we will derive a contradiction. The group $E$ is commutative and $p$-torsion, hence by Lemma \ref{unirimpcertmaps}, there exist $\alpha \in K - K^p$, $n \geq 1$, and a homomorphism $V_{n,\alpha} \rightarrow E$ whose image in $V_{1,\lambda}$ is nonzero. Stated differently, there is a nonzero homomorphism $g\colon V_{n,\alpha} \rightarrow V_{1,\lambda}$ such that $\delta(f\circ g) = 0$, or equivalently, $f\circ g\colon V_{n,\alpha} \rightarrow \Ga$ is of the form $F_{\mu}(h)$ for some homomorphism $h\colon V_{n,\alpha} \rightarrow \Ga^p$.

By Lemma \ref{mapVnalphtoVaalph} (and in the notation of that lemma), $K^p(\alpha) = K^p(\lambda)$, and
\begin{equation}
\label{cSnonzereqn}
f \circ g = cS \mbox{ for some } 0 \neq c \in K.
\end{equation}
Applying \cite[Prop.\,6.4]{rosmodulispaces}, there exist $p$-polynomials $X_0, \dots, X_{p-1} \in K[S, (S_j)_j]$ of degree $\leq 1$ in $S$ such that
\begin{equation}
\label{cSeqn}
f\circ g = cS = -X_{p-1} + \sum_{i=0}^{p-1} \mu^iX_i^p
\end{equation}
as homomorphisms $V_{n,\alpha} \rightarrow \Ga$. Applying \cite[Prop.\,6.4]{rosmodulispaces} once more, if we use the equation (\ref{eqnforVnalpha}) to eliminate $S^p$ from the $X_i^p$ terms above, then this becomes an identity of polynomials in $S$ and the $S_j$.
Write $$X_i = r_iS + \sum_j r_{ij}S_j^{p^{n-1}} + G_i$$ with $r_i, r_{ij} \in K$, and where $G_i$ does not involve either $S$ or $S_j^{p^{n-1}}$. We claim that $G_i$ has degree $< p^{n-1}$. Indeed, if not, then looking at the leading term in some $S_j$ in (\ref{cSeqn}) would yield a nonzero solution over $K$ to the equation $\sum_{i=0}^{p-1}\mu^ix_i^p = 0$, which would violate the fact that $\mu \notin K^p$. 

Comparing coefficients of $S_0^{p^n}$ in (\ref{cSeqn}) now yields
\[
\alpha^{1-p}\sum_{i=0}^{p-1} \mu^ir_i^p = \sum_{i=0}^{p-1}\mu^ir_{i0}^p.
\]
If some $r_i$ is nonzero, then $\alpha \in K^p(\mu)$. Since $\alpha \in K^p(\lambda)$, the $p$-independence of $\lambda, \mu$ would then imply that $\alpha \in K^p$, which is false. Thus $r_i = 0$ for all $0 \leq i < p$. It follows that each $X_i$ is a $p$-polynomial in the $S_j$ only, hence by (\ref{cSeqn}), $c = 0$, in violation of (\ref{cSnonzereqn}). This contradiction shows that $E$ is not unirational.
\end{example}

\begin{example}
\label{unirnotinherexts}
Let $K$ be a field of degree of imperfection $>1$, and let $\lambda, \mu \in K$ be $p$-independent elements. Assume first that $p > 2$. We will sketch how to modify the example and argument below when $p = 2$. Define $U \subset \Ga^{p+1}$ to be the $K$-group scheme
\begin{equation}
\label{Ueqnunirnotinherexts}
U := \left\{-X_{p-1} + \mu Y^p + \sum_{i=0}^{p-1}\lambda^iX_i^{p} = 0\right\}.
\end{equation}
The map $(Y, X_0, \dots X_{p-1}) \mapsto Y$ defines a surjective homomorphism $U \rightarrow \Ga$ with kernel $V_{1,\lambda}$. Thus $U$ is an extension of $\Ga$ by the unirational group $V_{1,\lambda}$, but we claim that $U$ is not unirational.

We first note that $K/\F_p(\lambda, \mu)$ is separable, precisely because $\lambda, \mu$ are $p$-independent over $K$ \cite[Th.\,26.6]{matsumura}.
Thus the non-unirationality of $U$ may be checked over $\F_p(\lambda, \mu)$ \cite[Th.\,1.6]{rosrigidity}, hence we may assume that $\lambda, \mu$ form a $p$-basis for $K$. We may also extend scalars to $K_s$ and thereby additionally assume that $K$ is separably closed. Assume for the sake of contradiction that $U$ {\em is} unirational. Then by Lemma \ref{unirimpcertmaps}, there exist $n > 0$,
\begin{equation}
\label{alphanotinK^pex}
\alpha \in K - K^p,
\end{equation}
and a homomorphism $g\colon V_{n,\alpha} \rightarrow U$ whose projection onto the $Y$ coordinate is nonzero. By \cite[Prop.\,6.4]{rosmodulispaces}, $g = (Y(S,(S_j)_j), X_0(S,(S_j)_j), \dots, X_{p-1}(S,(S_j)_j))$, where
\begin{align}
\label{eqnforYXi}
Y & = c\cdot S + F((S_j)_j) \nonumber \\
X_i &= c_i\cdot S + F_i((S_j)_j)
\end{align}
for some $c, c_i \in K$ and some $p$-polynomials $F, F_i$, and
\begin{equation}
\label{Yneq0}
Y(S,(S_j)_j) \neq 0.
\end{equation}
These expressions must satisfy the defining equation (\ref{Ueqnunirnotinherexts}) for $U$.

Substituting the expressions (\ref{eqnforYXi}) into (\ref{Ueqnunirnotinherexts}), and using the equation (\ref{eqnforVnalpha}) for $V_{n,\alpha}$ to eliminate $S^p$, we obtain
\begin{equation}
\label{S_ieqn}
-c_{p-1}S + \alpha^{1-p}\left(\mu c^p + \sum_{i=0}^{p-1} \lambda^ic_i^p\right)\left[S - \sum_{\substack{0 \leq j < p^n \\ j \not\equiv -1\pmod{p}}}\alpha^jS_j^{p^n} \right] - F_{p-1} + \mu F^p + \sum_{i=0}^{p-1}\lambda^iF_i^p = 0.
\end{equation}
Because the above expression has degree $< p$ in $S$, it is an identity of $p$-polynomials in the variables $S, S_j$ \cite[Prop.\,6.4]{rosmodulispaces}. We claim that 
\begin{equation}
\label{degF_p-1bd1}
\mathrm{deg}(F_{p-1}) \leq p^{n-1}. 
\end{equation}
Indeed, if not, then we have $p^d := \mathrm{deg}(F_{p-1}) > p^{n-1}$. Let $\beta_{ij}$ be the coefficient of $S_j^{p^{d}}$ in $F_i$, and let $\beta_j$ be the coefficient of $S_j^{p^d}$ in $F$, so $\beta_{p-1,j} \neq 0$ for some $j$. Then comparing coefficients of $S_j^{p^{d+1}}$ in (\ref{S_ieqn}) yields $\mu\beta_j^p + \sum_{i=0}^{p-1}\lambda^i\beta_{ij}^p = 0$, and the $p$-independence of $\lambda, \mu$ then shows that all $\beta_{ij} = 0$, in violation of the nonvanishing of $\beta_{p-1,j}$. Thus (\ref{degF_p-1bd1}) holds.

Let $r_0 := \mu c^p + \sum_{i=0}^{p-1}\lambda^ic_i^p$. We claim that $r_0 \neq 0$. Indeed, suppose that $r_0 = 0$ and we will obtain a contradiction. We have $c = 0$ by $p$-independence of $\lambda, \mu$, hence $F \neq 0$ by (\ref{Yneq0}) and (\ref{eqnforYXi}). Additionally, (\ref{S_ieqn}) would then yield
\begin{equation}
\label{falseSieqn}
-c_{p-1}S - F_{p-1} + \mu F^p + \sum_{i=0}^{p-1}\lambda^iF_i^p \stackrel{?}{=} 0.
\end{equation}
Let $p^d$ denote the maximum degree in of $F, F_i$ (well-defined since $F \neq 0$), and suppose that one of $F,F_i$ involves a term $S_j^{p^d}$ with nonzero coefficient. Then comparing the coefficients of $S_j^{p^{d+1}}$ in (\ref{falseSieqn}) yields a nontrivial solution over $K$ to the equation $\mu a^p + \sum_{i=0}^{p-1}\lambda^ia_i^p = 0$, in violation of the $p$-independence of $\lambda, \mu$. This contradiction shows that indeed $r_0 \neq 0$.

Let $$R := \sum_{i=0}^{p-1}\lambda^i(K^p + \mu K^p) \subset K,$$ and note that $r_0 \in R$. Comparing coefficients of $S_j^{p^n}$ in (\ref{S_ieqn}), and using (\ref{degF_p-1bd1}), we find that
\[
\alpha^{1-p+j} \in r_0^{-1}R, \hspace{.3 in} 0 \leq j < p^n, j \not\equiv -1 \pmod{p}.
\]
Since $r_0^{-1}R$ is closed under multiplication by $K^p$, we also have $\alpha^{1+j} \in r_0^{-1}R$, hence $\alpha^j \in r_0^{-1}R$ for all $0 < j < p$. Since $1 \in r_0^{-1}R$, and $r_0^{-1}R$ is closed under multiplication by $\lambda$, it follows that $K^p[\alpha, \lambda] \subset r_0^{-1}R$. We claim that $\alpha \in K^p(\lambda)$. For if not, then -- because $K$ has degree of imperfection $2$ -- it would follow that $K = K^p[\alpha, \lambda] \subset r_0^{-1}R$, hence $R = K$. But $\mu^2 \notin R$ due to the $p$-independence of $\lambda, \mu$. (Here we use $p > 2$.) Thus 
\begin{equation}
\label{alphainK^plambda}
\alpha \in K^p(\lambda),
\end{equation}
as claimed.

We claim that $F$ is homogeneous of degree $p^{n-1}$. To prove this, we first extend scalars to $L_1:= K(\mu^{1/p})$. Taking $Z, Z_j$ (instead of $S, S_j$) to be the variables on $V_{1,\lambda}$, we have the map $h_1\colon U_{L_1} \rightarrow V_{1,\lambda}$ defined by the formula $(Y, (X_i)_i) \mapsto (Z, (Z_j)_j)$ where $Z := X_{p-1}$, $Z_j := X_j$ for $0 < j < p-1$, and $Z_0 := X_0+\mu^{1/p}Y$. Because $\alpha \notin L_1^p = K^p(\mu)$, due to (\ref{alphanotinK^pex}) and the equality $K^p(\lambda) \cap K^p(\mu) = K^p$, we may Lemma \ref{mapVnalphtoVaalph} to the map $h_1\circ g$, in conjunction with \cite[Prop.\,6.4]{rosmodulispaces} to conclude that $(X_0 + \mu^{1/p}Y)(S, (S_j)_j)$ is homogeneous of degree $p^{n-1}$ in the $S_j$ (by which we mean once one excludes the linear term involving $S$). Now we extend scalars to $L_2 := K((\mu/\lambda)^{1/p})$. Over $L_2$, we have the map $h_2\colon U_{L_2} \rightarrow V_{1,\lambda}$ defined by the formula $Z := X_{p-1}$, $Z_j := X_j$ for $j = 0$ and $1 < j < p-1$, and $Z_1 := X_1 + (\mu/\lambda)^{1/p}Y$. (Here we again use $p > 2$ to ensure that $1 < p-1$.) Once more, we may apply Lemma \ref{mapVnalphtoVaalph} to $h_2\circ g$ to conclude that $Z_0 = X_0$ is homogeneous of degree $p^{n-1}$ in the $S_j$. Since $X_0+\mu^{1/p}Y$ was as well, we deduce that $Y(S, (S_j)_j)$ is homogeneous of degree $p^{n-1}$ in those variables, hence $F$ is (see (\ref{eqnforYXi})). This proves the claim.

Now we are ready to obtain our contradiction and thereby complete the proof that $U$ is not unirational. Write $F = \sum_j \gamma_jS_j^{p^{n-1}}$. Comparing coefficients of $S_j^{p^n}$ in (\ref{S_ieqn}), and using (\ref{degF_p-1bd1}), we find that $\mu(\gamma_j^p - \alpha^{1-p+j}c^p) \in K^p[\alpha, \lambda]$. If $\gamma_j^p-\alpha^{1-p+j}c^p \neq 0$ for some $j$, then it follows that $\mu \in K^p[\alpha, \lambda] \subset K^p[\lambda]$ by (\ref{alphainK^plambda}), a contradiction. Thus we must have $\gamma_j^p -\alpha^{1-p+j}c^p = 0$ for all $j$. Taking $j = 0$, we must have $c = 0$ because $\alpha \notin K^p$. Thus $\gamma_j = 0$ for all $j$, hence $F = 0$, in violation of (\ref{Yneq0}) and (\ref{eqnforYXi}).
\end{example}

\begin{remark}
When $p = 2$, the above argument does not work, and indeed, in this case $U$ {\em is} unirational, being a smooth quadric hypersurface (with a rational point). But if one instead takes the $K$-group 
\begin{equation}
\label{rmkunirextschar2eqn1}
U' := \{\mu Y^4 - X_1 + X_0^2 + \lambda X_1^2 = 0\} \subset \Ga^3,
\end{equation}
then projection onto $Y$ defines a surjection $U' \rightarrow \Ga$ with unirational kernel, but $U'$ is not unirational. We merely sketch the proof of this last claim. Using the notation of (\ref{eqnforYXi}), one writes out the conditions for $Y, X_j$ to satisfy the defining equation (\ref{rmkunirextschar2eqn1}) for $U'$. One first shows that $\mathrm{deg}(F_1) \leq 2^n$ as above, and then compares leading coefficients of $S_j^{2^{n+1}}$ and uses the $2$-independence of $\lambda, \mu$ and the fact that $\alpha \notin K^2$ to conclude that $c = 0$. Then one shows that in fact $\mathrm{deg}(F_1) \leq 2^{n-1}$. Because of the surjections $V_{n+1,\alpha} \twoheadrightarrow V_{n,\alpha}$ constructed in the proof of Lemma \ref{mapVnalphtoVaalph}, one may assume that $n \geq 2$. Comparing coefficients of $S_j^{2^n}$, one then shows that $\alpha, \alpha^{-1} \in r_0^{-1}R'$, where $R' := \mu K^4 + K^2(\lambda)$ for some $0 \neq r_0 \in K^2(\lambda)$. One checks that this implies that in fact $\alpha \in K^2(\lambda)$. Then, analogously to the argument above, one passes to the extensions $L_1 := K(\mu^{1/2})$ and $L_2 := K((\mu/\lambda)^{1/4})$ and makes suitable changes of variables (same as in Example \ref{unirnotinherexts} for $L_1$, but in the case of $L_2$, the change is $(Y, X_0, X_1) \mapsto (Y, X_0 + (\mu/\lambda)^{1/4}Y, X_1 + (\mu/\lambda)^{1/2}Y^2)$) to apply Lemma \ref{mapVnalphtoVaalph} and deduce that $F$ is homogeneous of degree $2^{n-2}$. Finally, one concludes as in the example above that in fact $F = 0$, hence $Y = 0$. 
\end{remark}

\section{Ext-unirational groups}
\label{extunirsec}

The examples of the preceding section demonstrate that unirationality is not inherited by extensions over fields of degree of imperfection $>1$. In light of this failure, it is natural to make the following definition.

\begin{definition}
\label{extunirational}
We say that a smooth $K$-group scheme $G$ is {\em ext-unirational} when it admits a filtration $$1 = G_0 \trianglelefteq G_1 \trianglelefteq \dots \trianglelefteq G_n = G$$ such that $G_{i+1}/G_i$ is unirational for all $0 \leq i < n$.
\end{definition}

Ext-unirational groups are necessarily affine and connected, because this is true of unirational groups. By Theorem \ref{extunirdegimp1}, when $K$ has degree of imperfection $1$ (also, trivially, when $K$ is perfect), ext-unirationality is the same as unirationality, but over fields of larger degree of imperfection, ext-unirationality is a strictly weaker condition. Note also that ext-unirationality (like unirationality) is inherited by quotients and (unlike unirationality in general) by extensions. As usual, the notion of ext-unirationality is only interesting over imperfect fields, as over perfect fields every connected linear algebraic group is unirational.

Ext-unirationality may be rephrased more canonically as follows. Given a smooth finite type $K$-group scheme $G$, let $G_{\mathrm{uni}} \subset G$ denote its maximal unirational $K$-subgroup scheme. Then $G_{\mathrm{uni}}$ is a normal subgroup of $G$ \cite[Cor.\,7.11]{rosrigidity}. Now define a sequence of $K$-group quotients $G_u^{[n]}$ of $G$ recursively by the formulas $$G_u^{[0]} := G, \hspace{.3 in} G_u^{[n+1]} := G_u^{[n]}/(G_u^{[n]})_{\uni}, n \geq 0.$$ Then $G$ is ext-unirational if and only if $G_u^{[n]} = 1$ for some (equivalently, all sufficiently large) $n \geq 0$. This also shows that we may replace the condition in Definition \ref{extunirational} that each $G_i$ be normal in $G_{i+1}$ by the condition that it is normal in $G$ without altering the definition. Furthermore, the construction of the $G_u^{[n]}$ commutes with separable extension of the ground field \cite[Cor.\,7.10]{rosrigidity}, and as a consequence, ext-unirationality is insensitive to passage to separable extensions:

\begin{proposition}
If $L/K$ is a separable field extension, then $G$ is ext-unirational over $K$ if and only if it is so over $L$.
\end{proposition}

The following lemma will be useful for proving results beyond the affine setting.

\begin{lemma}
\label{abvarquotaffker}
If $G$ is a connected $K$-group scheme, then there is a surjection $f\colon G \twoheadrightarrow A$ such that $A$ is an abelian variety and $\ker(f)$ is affine.
\end{lemma}

\begin{proof}
If $\mathrm{char}(K) = 0$, then we are done by Chevalley's Theorem. If $\mathrm{char}(K) > 0$, then $G/I$ is smooth for some infinitesimal $K$-subgroup scheme $I \trianglelefteq G$ \cite[VII$_{\rm{A}}$, Prop.\,8.3]{sga3}, so we may assume that $G$ is smooth. By Chevalley again, $G$ admits a map as in the lemma over $K_{\perf}$, the perfect closure of $K$. Thus there is such a quotient over $G_{K^{1p^n}}$ for some $n > 0$.
We may identify $K^{1/p^n}$ with $K$ via the $p^n$-power map, and then $G_{K^{1/p^n}}$ becomes identified with $G^{(p^n)}$. Thus $G^{(p^n)}$ admits such a map. But because $G$ is smooth, $G^{(p^n)}$ is the quotient of $G$ by the $n$th order relative Frobenius map over $K$, so $G$ also admits an abelian variety quotient with affine kernel.
\end{proof}

For a smooth connected $K$-group scheme $G$, we denote its derived group by $\mathscr{D}G$ and its abelianization $G/\mathscr{D}G$ by $G^{\ab}$. Next we verify that ext-unirationality reduces to the $p$-torsion commutative setting.

\begin{proposition}
\label{extunirunipptorcom}
A smooth connected $K$-group scheme $G$ is ext-unirational if and only if its maximal commutative quotient $G^{\ab}$ is. If $G$ is affine and $\mathrm{char}(K) = p > 0$, then this is further equivalent to ext-unirationality of its maximal $p$-torsion commutative quotient $G^{\ab}/[p]G^{\ab}$.
\end{proposition}

\begin{proof}
When $\mathrm{char}(K) = 0$, ext-unirationality is equivalent to affineness, so in this case it suffices to note that affineness of $G^{\ab}$ implies the same for $G$, a consequence of Lemma \ref{abvarquotaffker} because abelian varieties are commutative. Since ext-unirationality implies affineness, this also reduces the assertion for smooth connected $G$ to the affine setting in characteristic $p$. Thus we may assume that $G$ is affine and that $\mathrm{char}(K) = p > 0$, and we wish to show that ext-unirationality of $G^{\ab}/[p]G^{\ab}$ implies the same for $G$.

We may assume that $K$ is separably closed. The proof is by dimension induction. For any quotient $\overline{G}$ of $G$, the map $G^{\ab}/[p]G^{\ab}\rightarrow \overline{G}^{\ab}/[p]\overline{G}^{\ab}$ is surjective, so if $\mathrm{dim}(\overline{G}) < \mathrm{dim}(G)$, then we may assume that $\overline{G}$ is ext-unirational. First suppose that $G = U$ is commutative and unipotent. If $U$ is $p$-torsion, then the proposition is immediate, so assume it is not, and let $n \geq 1$ be such that $[p^{n+1}]U = 0$ but $[p^n]U \neq 0$. Let $\overline{U} := U/[p^n]U$. Then $0 < \mathrm{dim}(\overline{U}) < \mathrm{dim}(U)$, and in particular, $\overline{U}$ is ext-unirational. The (nonzero) map $[p]\colon U \rightarrow U$ factors through a nonzero map $[p]\colon \overline{U} \rightarrow U$ whose image is ext-unirational. Since the image is nonzero, the cokernel is also ext-unirational, hence so is $U$. This completes the proof when $U$ is commutative unipotent.

Now suppose that $G = U$ is non-commutative unipotent, so the commutator map $c\colon U^2 \rightarrow U$ is nonconstant. Let $\overline{U}$ be a maximal dimensional quotient of $U$ such that $c$ factors through a map $\overline{c}\colon \overline{U} \times U \rightarrow U$. Because $c$ is nonconstant, $\overline{U} \neq 1$. On the other hand, because $U$ is nilpotent, it admits a nontrivial smooth connected central $K$-subgroup $U'$, and the commutator map factors through a map $U/U' \times U \rightarrow U$, hence $\mathrm{dim}(\overline{U}) < \mathrm{dim}(U)$. In particular, $\overline{U}$ is ext-unirational, so $\overline{U}_{\uni} \neq 1$. If the map $\overline{U}_{\uni} \rightarrow U$ defined by the formula $u \mapsto c(u_1u, u_2)c(u_1,u_2)^{-1}$ is constant for all $u_1 \in \overline{U}(K)$, $u_2 \in U(K)$, then -- because $K$ is separably closed and $\overline{U}, U$ are smooth -- it would follow that $c$ factors through a map $\overline{U}/\overline{U}_{\uni} \times U \rightarrow U$, in violation of the minimality of $\mathrm{dim}(\overline{U})$. Thus one can find $u_1, u_2$ such that this map is nonconstant, so we have a nonconstant map from a positive-dimensional unirational $K$-scheme into $U$. It follows that $U_{\uni} \neq 1$, hence $U/U_{\uni}$ is ext-unirational, hence so is $U$. This completes the proof of the proposition for unipotent $G$.

Next we treat smooth connected affine $G$, where we assume that $G^{\ab}/[p]G^{\ab}$ is ext-unirational, and we wish to prove the same for $G$. Let $G_t \subset G$ denote the $K$-subgroup generated by the $K$-tori of $G$. This is a normal $K$-subgroup, and $U := G/G_t$ is unipotent \cite[Prop.\,A.2.11]{cgp}. Furthermore, because tori are unirational, $G$ is ext-unirational if and only if $U$ is. The map $G^{\ab}/[p]G^{\ab} \rightarrow U^{\ab}/[p]U^{\ab}$ is surjective (in fact, it is an isomorphism) and ext-unirationality is inherited by quotients, so the general case follows from the unipotent one.
\end{proof}

When $K$ has degree of imperfection $1$, the above proposition has the following nice consequence. 

\begin{corollary}
\label{uniriffabis}
Let $K$ be a field of degree of imperfection $1$. A smooth connected $K$-group scheme $G$ is unirational if and only if its abelianization $G^{\mathrm{ab}}$ is. If $G$ is also affine, then $G$ is unirational if and only if $G^{\ab}/[p]G^{\ab}$ is.
\end{corollary}

\begin{proof}
Combine Proposition \ref{extunirunipptorcom} and Theorem \ref{extunirdegimp1}.
\end{proof}

We will discuss analogues of Corollary \ref{uniriffabis} over fields of higher degree of imperfection in the next section.

Let $K$ be a field, and suppose given two smooth connected $K$-group schemes $G, H$. Let $\phi\colon G \dashrightarrow H$ be a rational map which is birational, with inverse $\psi\colon H \dashrightarrow G$. Let $U \subset G$ denote the open locus of definition of $\phi$. Denote by $f$ the composition $U \xrightarrow{\phi} H \xrightarrow{\pi_H} H/H_{\uni}$. We claim that $f$ factors through the open set $\pi_G(U) \subset G/G_{\uni}$, where $\pi_G\colon G \rightarrow G/G_{'\uni}$ is the (flat) quotient map. For this assertion, we may assume that $K$ is separably closed, as the formation of the maximal unirational subgroup scheme is insensitive to separable extension. Let $a\colon U \times G_{\uni} \rightarrow G$ denote the action map $(u, x) \mapsto ux$, and let $V := a^{-1}(U)$. By descent, $f$ factors through $\pi_G(U)$ if and only if the two maps $V \rightarrow H/H_{\uni}$, $(u, g) \mapsto \phi(u)$, $(u, g) \mapsto \phi(ug)$ coincide. We may verify this upon restriction to $V_{u_0} := V \cap (u_0 \times G_{\uni})$ for each $u_0 \in U(K)$. Note that $V_{u_0}$ is open in $G_{\uni}$, hence unirational. We must show that the composition $V_{u_0} \rightarrow H \rightarrow H/H_{\uni}$ of $\pi_H \circ \phi$ is constant (where $\pi_H\colon H \rightarrow H/H_{\uni}$ is the quotient map), and this follows from the fact that the image of $\phi|V_{u_0}$ is a unirational subscheme of $H$. This proves that $f$ factors through a map $G/G_{\uni} \rightarrow H/H_{\uni}$. We may similarly factor the inverse map $\psi$ to get a map $H/H_{\uni} \rightarrow G/G_{\uni}$, and the compositions in both directions are the identity. We thereby deduce that $G_{\uni}$ is nontrivial if and only if $H_{\uni}$ is, and that $G/G_{\uni}$ is birationally equivalent to $H/H_{\uni}$. By dimension induction, therefore, $G$ is ext-unirational if and only if $H$ is. We have proved the following proposition.

\begin{proposition}
\label{extunibirinvt}
If $G, H$ are smooth connected $K$-group schemes that are birationally equivalent, then $G$ is ext-unirational if and only if $H$ is.
\end{proposition}

If $T \subset G$ is a torus in a connected linear algebraic $K$-group $G$, then $G$ is unirational if and only if $Z_G(T)$ is \cite[Prop.\,7.12]{rosrigidity}. Applying this with $T$ a maximal torus, and using Lemma \ref{extunitor}, one sees that unirationality of $G$ is equivalent to that of the unipotent group $Z_G(T)/T$, whence of its maximal wound quotient. The same statement holds true for ext-unirationality, thanks to the following analogous result.

\begin{proposition}
\label{Gextunir}
For a connected linear algebraic $K$-group $G$, the following are equivalent:
\begin{itemize}
\item[(i)] $G$ is ext-unirational.
\item[(ii)] $Z_G(T)$ is ext-unirational for every $K$-torus $T \subset G$.
\item[(iii)] $Z_G(T)$ is ext-unirational for some $K$-torus $T \subset G$.
\end{itemize}
\end{proposition}

\begin{proof}
We may assume that $K$ is separably closed. Let $T \subset G$ be a torus. For a generic cocharacter $\lambda\colon \Gm \rightarrow T$ (lying in the complement of the union of finitely many hyperplanes in the cocharacter lattice), one has $Z_G(T) = Z_G(\lambda)$. Then there are (split) unipotent $K$-subgroups $U^+, U^- \subset G$ such that the map $$U^- \times Z_G(T) \times U^+ \rightarrow G, \hspace{.3 in} (x, y, z) \mapsto xyz$$ is an open embedding \cite[Lem.\,2.1.5,Props.\,2.1.8(2)(3),2.1.10]{cgp}. By Proposition \ref{extunibirinvt}, $G$ is ext-unirational if and only if $U^- \times Z_G(T) \times U^+$ is, which in turn holds if and only if $Z_G(T)$ is.
\end{proof}

\section{Deducing unirationality from quotients}
\label{dedfromquotssec}

Corollary \ref{uniriffabis} says that, when $K$ has degree of imperfection $1$, unirationality of a group may be tested upon passage to a certain quotient of that group. In this section we will prove an analogue of this result for unipotent groups $U$ over fields of higher degree of imperfection. In particular, we will see that unirationality of $U^{\ab}$ implies unirationality of $U$ when $K$ has degree of imperfection $\leq 2$. We will also show that this result is optimal by giving, over every field of degree of imperfection $\geq 3$, an example of a wound unipotent group $U$ such that $U^{\ab}$ is $p$-torsion and unirational even though $U$ fails to be unirational (though $U$ must be ext-unirational by Proposition \ref{extunirunipptorcom}).

We first verify an analogous result for permawoundness in place of unirationality. For the definition of permawoundness, see either \cite[Def.\,1.2]{rospermawound} or the beginning of \S\ref{degimp1sec}. The proof below also makes reference to the notion of semiwoundness. Recall that a unipotent $K$-group scheme $U$ is called semiwound when it contains no copies of $\Ga$ over $K$. This notion exhibits similar properties to the closely related notion of woundness, which is merely semiwoundness plus smoothness and connectedness; see \cite[Appendix A]{rospermawound} for the basic properties of semiwound groups.

\begin{proposition}
\label{pmwdfollfrmquot}
A smooth unipotent $K$-group scheme $U$ is permawound if and only if $U^{\ab}/[p]U^{\ab}$ is.
\end{proposition}

\begin{proof}
All smooth unipotent groups are permawound over perfect fields \cite[Prop.\,5.2(i)]{rospermawound}, so we may assume that $K$ is imperfect. The only if direction follows from the fact that permawoundness is inherited by quotients, so we now concentrate on the converse. Let us first assume that $U$ is commutative (which, though we do not know it yet, must be the case a posteriori when $U$ is wound, as all wound permawound unipotent groups are commutative). The group $U$ is killed by $p^n$ for some $n \geq 0$, and we proceed by induction on $n$, the $n = 0$ case being trivial. So suppose that $n > 0$. The map $[p^{n-1}]\colon U \rightarrow U$ descends to a surjective map $U/[p]U \twoheadrightarrow [p^{n-1}]U$, hence $[p^{n-1}]U$ is permawound. If we let $\overline{U} := U/[p^{n-1}]U$, then $U/[p]U \twoheadrightarrow \overline{U}/[p]\overline{U}$, so the latter group is permawound, hence, by induction, so is $\overline{U}$, because it is killed by $p^{n-1}$. Thus $U$ an extension of the two permawound groups $\overline{U}$ and $[p^{n-1}]U$, hence is itself permawound by Proposition \ref{permawdexts}.

Now let $U$ be a semiwound unipotent group such that $U^{\ab}/[p]U^{\ab}$ is permawound. We will show that $U$ is commutative, hence permawound by the already-treated commutative case. First we note that $U$ must be connected, as otherwise $U$ would admit a nontrivial \'etale unipotent quotient, hence a nontrivial commutative $p$-torsion \'etale quotient, which would imply that $U^{\ab}/[p]U^{\ab}$ is disconnected, in violation of \cite[Prop.\,6.2]{rospermawound}. If $U = 1$ the assertion is immediate, so assume that $U \neq 1$. Then $U$ contains a nontrivial smooth connected central $K$-subgroup $U' \subset U$. By dimension induction, the quotient $\overline{U} := U/U'$ is permawound, as $U^{\ab}/[p]U^{\ab}$ surjects onto $\overline{U}^{\ab}/[p]\overline{U}^{\ab}$. The commutator map $U^2 \rightarrow U$ descends to a map $\overline{U}^2 \rightarrow U$, which further descends to a map $\overline{U}_w^2 \rightarrow U$, where $\overline{U}_w$ is the maximal wound quotient of $\overline{U}$. This map is constant by \cite[Cor.\,10.3,Th.\,1.9(ii)]{rosrigidity}, hence $U$ is commutative, as claimed.

Finally, let $U$ be an arbitrary smooth unipotent $K$-group scheme such that $U^{\ab}/[p]U^{\ab}$ is permawound, and let $U_s \trianglelefteq U$ be the maximal split $K$-subgroup. Then $U/U_s$ is permawound by the semiwound case, hence so is $U$, by Proposition \ref{permawdexts} and \cite[Prop.\,5.3]{rospermawound}.
\end{proof}

We will require the following result, of interest in its own right, that gives a finiteness criterion for permawoundness.

\begin{proposition}
\label{permawdfinite}
Let $K$ be a field of finite degree of imperfection. A smooth unipotent $K$-group scheme $U$ is permawound if and only if, for every commutative $p$-torsion wound unipotent $K$-group $V$, the group $\Hom_{K_s}(U, V)$ is finite.
\end{proposition}

\begin{proof}
Thanks to Proposition \ref{pmwdfollfrmquot}, we may assume that $U$ is commutative and $p$-torsion. When $K$ is perfect, the only wound group is the trivial group, and every smooth unipotent $K$-group is permawound \cite[Prop.\,5.2(i)]{rospermawound}, so we may assume that $K$ is imperfect. If $U_s \trianglelefteq U$ denotes the maximal split $K$-subgroup, then $U$ is permawound if and only if $U/U_s$ is, by Proposition \ref{permawdexts} and \cite[Prop.\,5.3]{rospermawound}. Since $U_s$ is killed under any $K_s$-homomorphism from $U$ into a wound group, we may assume $U$ is semiwound (that is, contains no copy of $\Ga$). The only if direction now follows from \cite[Prop.\,10.1]{rosrigidity}, plus the connectedness of $U$ (to ensure woundness of the maximal wound quotient) \cite[Prop.\,6.2]{rospermawound}.

Now suppose that $K$ has finite nonzero degree of imperfection, and assume that $U$ is not permawound. We will show that $\Hom_{K_s}(U, \mathscr{V})$ is infinite, where (by abuse of notation) $\mathscr{V}$ is any $K$-form of the $K_s$-group which is denoted by $\mathscr{V}$ in \cite[Def.\,7.3]{rospermawound}. By \cite[Th.\,1.4]{rospermawound}, we have for some $n > 0$ ($>0$ because $U$ is not permawound) an exact sequence $$0 \longrightarrow U \longrightarrow W \longrightarrow \Ga^n \longrightarrow 0$$ with $W$ wound, permawound, commutative, and $p$-torsion. We will show that $\Ext^1_{K_s}(W, \mathscr{V})$ is finite but $\Ext^1_{K_s}(\Ga, \mathscr{V})$ is infinite, which will imply that $\Hom_{K_s}(U, \mathscr{V})$ is infinite as well. We may assume that $K= K_s$ \cite[Cor.\,10.5]{rosrigidity}.

First we prove the finiteness of $\Ext^1(W, \mathscr{V})$. We have a homomorphism $\phi\colon \Ext^1(W, \mathscr{V}) \rightarrow \Hom(W, \mathscr{V})$ defined as follows: Given an extension $E$ in the former group, the Verschiebung endomorphism of $E$ descends to a map $W \rightarrow \mathscr{V}$ from the quotient $W$ to the subgroup $\mathscr{V}$. Then $\ker(\phi)$ consists of the extensions with trivial Verschiebung, and this is a trivial group by \cite[Cor.\,8.3]{rospermawound}. Thus the desired finiteness follows from the finiteness of $\Hom(W, \mathscr{V})$.

It remains to prove the infinitude of $\Ext^1(\Ga, \mathscr{V})$. This group is a $K$-vector space via the action of $K$ on $\Ga$, so it is equivalent to show that it is nonzero. By \cite[Prop.\,B.1.13]{cgp}, we have an exact sequence
\begin{equation}
\label{permawdfinitepfeqn3}
0 \longrightarrow \mathscr{V} \longrightarrow \Ga^{d+1} \xlongrightarrow{F} \Ga \longrightarrow 0,
\end{equation}
for a suitable $p$-polynomial $F$, where $d := \mathrm{dim}(\mathscr{V})$. The sequence (\ref{permawdfinitepfeqn3}) does not split, because the wound $\mathscr{V}$ cannot be a quotient of $\Ga^{d+1}$. Thus this sequence yields a nonzero element of $\Ext^1(\Ga, \mathscr{V})$, as required.
\end{proof}

\begin{proposition}
\label{extunirgpspmwdunir}
If $K$ is imperfect, then every extension of a unirational $K$-group scheme by a permawound unipotent $K$-group $U$ is unirational.
\end{proposition}

\begin{proof}
We may assume that $K$ is separably closed. If $U$ is not semiwound, then we may use the fact that $\mathrm{H}^1(X, \Ga) = 0$ for every affine scheme $X$ and dimension induction to conclude. We therefore assume that $U$ is semiwound. If $K$ has infinite degree of imperfection, then the assertion is immediate, thanks to \cite[Prop.\,6.3]{rospermawound}. Thus we may assume that $K$ has finite degree of imperfection. Because permawound groups are unirational \cite[Th.\,1.9]{rosrigidity}, we may use the rigidity property of permawound groups \cite[Th.\,1.5]{rospermawound} and an easy induction to reduce to showing that any torsor for the $K$-group $V$ over a unirational $K$-scheme $X$ is also unirational, where $V$ is either $\R_{K^{1/p}/K}(\alpha_p)$ or $\mathscr{V}$. But this assertion follows from \cite[Prop.\,8.6]{rosrigidity}.
\end{proof}

The proof of the following lemma uses the notion of the restricted moduli space of pointed morphisms $\calMor((X,x), (G,1))^+$ from a $K$-scheme into a $K$-group scheme. This is the functor $$\{\mbox{geometrically reduced $K$-schemes}\} \rightarrow \{\mbox{groups}\}$$ that sends a test scheme $T$ to the space of pointed $T$-morphisms $(X \times_K T, x_T) \rightarrow (G \times_K T, 1_T)$. It turns out that it is represented by a smooth unipotent $K$-group scheme when $X$ is geometrically reduced of finite type and $G$ is wound unipotent, and in fact exhibits the somewhat stronger property that the corresponding scheme may be regarded as a subfunctor of the above morphism functor considered on the category of {\em all} $K$-schemes \cite[Th.\,4.3]{rosmodulispaces}.

\begin{lemma}
\label{finmultaddmaps}
Let $d, n, r \geq 0$ be integers.
\begin{itemize}
\item[(i)] There is a constant $C = C(d, n, r) > 0$ with the following property: Let $K$ be a field of characteristic $p$ and degree of imperfection $r$, and let $U \simeq \{F = 0\}$ with $F$ a reduced $p$-polynomial over $K$ of degree $d$. Then for every $r$-tuple of divisors $D_1,\dots, D_r \subset \P^1_K$ with $\sum_i \mathrm{deg}(D_i) \leq n$, there are $\leq C$ multi-additive maps $\prod_{i=1}^r \R_{D_i/K}(\Gm) \rightarrow U$.
\item[(ii)] If $U$ is wound unipotent over a field $K$ of degree of imperfection $r$, then there is a constant $C = C(U, n) > 0$ such that, for every $r$-tuple of divisors $D_1,\dots, D_r \subset \P^1_K$ with $\sum_i \mathrm{deg}(D_i) \leq n$, there are $\leq C$ multi-additive maps $\prod_{i=1}^r \R_{D_i/K}(\Gm) \rightarrow U$.
\end{itemize}
\end{lemma}

\begin{proof}
Assertion (ii) follows from (i), since $U$ admits a filtration by commutative $p$-torsion wound unipotent groups by \cite[Prop.\,B.3.2]{cgp}, and every such group is of the form $\{F = 0\}$ for some reduced $p$-polynomial $F$ \cite[Prop.\,B.1.13]{cgp}. (We may assume that $K$ is infinite, since otherwise $r=0$.) So we concentrate on (i). We may assume that $K$ is separably closed. By \cite[Lem.\,5.4]{rosrigidity}, one may write $D_i$ as a finite disjoint union of $K$-schemes of the form $\mathrm{Spec}(K(\alpha_{ij}^{1/p^{n_{ij}}}))$ with $\alpha_{ij} \in K$, hence $\R_{D_i/K}(\Gm) \simeq \prod_j \R_{K(\alpha_{ij}^{1/p^{n_{ij}}})/K}(\Gm)$, so we are reduced to the case in which $D_i = \mathrm{Spec}(K(\lambda_i^{1/p^{n_i}}))$ for all $i$.

The group $G_i := \R_{K(\lambda_i^{1/p^{n_i}})}(\Gm)/\Gm$ is the so-called generalized Jacobian of the pair $(\P^1, D_i)$, where $D_i \subset \P^1$ denotes the reduced divisor with support $\lambda_i^{1/p^{n_i}}$. For an explanation of generalized Jacobians, see, for instance, the discussion beginning in the second paragraph of \cite[\S2]{rosmodulispaces}. Let $X_i := \P^1\backslash\{\lambda_i^{1/p^{n_i}}\}$. By \cite[Th.\,6.7]{rosmodulispaces}, the assertion of the lemma is equivalent to showing that there is $C = C(d, r, n_1, \dots, n_r)$ such that there are $\leq C$ maps $\prod_{i=1}^r X_i \rightarrow U$ that vanish whenever any of the coordinates is set to $\infty$.

We proceed by induction on $r$, the case $r = 0$ being trivial. Suppose that $r > 0$. By the ubiquity property of permawound groups \cite[Th.\,1.4]{rospermawound}, there is an exact sequence $$0 \longrightarrow U \longrightarrow W \longrightarrow \Ga^m \longrightarrow 0$$ with $W$ permawound. We require a slightly more precise result, namely, that $m$ may be bounded depending only on $r$ and $d$. (Note that $p = \mathrm{char}(K)$ is determined by $d$.) This follows from the proof of ubiquity. In particular, $\dim(W)$ may be bounded depending only on $r$ and $d$. By the rigidity property of permawound groups \cite[Th.\,1.5]{rospermawound}, together with the fact that the number of terms in a filtration as in that result is bounded in terms of $\dim(W)$ and $p$, we may in fact assume that $U$ is either $\R_{K^{1/p}/K}(\alpha_p)$ or $\mathscr{V}$ (even though technically the former is not wound, being nonsmooth). In the former case, $U$ is totally nonsmooth, so the only map from a geometrically reduced $K$-scheme into $U$ is the $0$ map. So assume that $U = \mathscr{V}$, and we must prove the existence of a suitable $C = C(p,r, n_1,\dots,n_r)$.

We may assume that all $n_i$ are as small as possible -- that is, if $n_i > 0$, then $\lambda_i \notin K^p$. If $n_i = 0$, then $X_1 \simeq \A^1$, so any map $\prod X_i \rightarrow U$ which vanishes on $\infty \times \prod_{i>1} X_i$ must be constant because $U$ is wound. Thus we may assume that $\lambda_1 \notin K^p$ and $n_1 > 0$. By \cite[Lem.\,5.3]{rosrigidity}, there is a totally nonsmooth (over $K$) $K$-subgroup $N \subset M := \calMor((\P^1\backslash \lambda_1^{1/p^{n_1}}, \infty), (\mathscr{V},0))^+$ such that $(M_{K(\lambda_1^{1/p^{\infty}})})_s \subset N_{K(\lambda_1^{1/p^{\infty}})}$, where, for a unipotent $L$-group $V$, $V_s$ denotes the maximal split $L$-subgroup of $V$.

Now a map $\prod_i X_i \rightarrow U$ which vanishes whenever any coordinate is $\infty$ is the same thing as such a map $g\colon \prod_{i>1} X_i \rightarrow M$, and we will show that the number of such maps is finite and bounded depending only on the $n_i$, $r$, and $p$. Let $L := K(\lambda_1^{1/p^{\infty}})$, a field of degree of imperfection $r-1$. Any map $g$ as above induces a map $\overline{g}\colon \prod_{i>1} X_i \rightarrow M_L/(M_L)_s$ which vanishes whenever some coordinate is $\infty$. By induction, there are at most $C' = C'(d', r-1, n_2, \dots, n_r)$ such maps, where $d'$ is the degree of an equation describing $M_L/(M_L)_s$. Further, we claim that if $\overline{g} = 0$, then $g = 0$. For if $\overline{g} = 0$, then $g$ lands in $N$, since $N$ contains $(M_L)_s$. Because $N$ is totally nonsmooth over $K$ and $\prod_{i>1} X_i$ is smooth, we must have $g = 0$, as claimed.

It only remains to show that $d'$ is bounded in terms only of $n_1, r, p$. For this we argue as follows. Complete $\lambda_1$ to a $p$-basis $\mu_1 := \lambda_1, \mu_2, \dots, \mu_r$ of $K$, so in particular the $\mu_i$ are $p$-independent. Consider the subfield $F := \F_p(\mu_1, \dots, \mu_r) \subset K$. Then $K/F$ is a separable extension, hence the formation of $M = \calMor_F((\P^1\backslash \mu_1^{1/p^{n_1}}, \infty), (\mathscr{V},0))^+$ commutes with scalar extension to $K$, where by abuse of notation, $\mathscr{V}$ is an $F$-descent of $\mathscr{V}/K$ (which exists because $\mathscr{V}$ may be defined via an equation using the $p$-basis $\mu_1, \dots, \mu_r$ of $F$). Furthermore, if we let $E := F(\lambda_1^{1/p^{\infty}})$, then $L/E$ is separable, hence $(M_L)_s = (M_E)_s \times_E L$. Thus we have reduced the calculation of $M_L/(M_L)_s$ to the ``universal'' case over the single field $F$ with fixed $p$-basis $\mu_1, \dots, \mu_r$. Thus the $d'$ that works over this field also works in general, and has no dependence on the chosen $\lambda_i \in K$. The proof of the lemma is complete.
\end{proof}

\begin{proposition}
\label{finmapsprodr}
Let $K$ be a field of degree of imperfection $r < \infty$, let $(X_1, x_1), \dots, (X_r, x_r)$ be pointed unirational $K$-schemes with $x_i \in X_i(K)$, and let $U$ be a wound unipotent $K$-group. Then there are only finitely many maps $f\colon\prod_{i=1}^r X_i \rightarrow U$ such that $f$ vanishes whenever any one of the coordinates is set to $x_i$.
\end{proposition}

\begin{proof}
We first treat the special case in which the $X_i$ are all open subschemes of $\P^1_K$, in which case we prove a somewhat stronger result -- namely, that the number of maps $f$ as in the proposition is bounded in terms only of $U$ and the degrees of the complementary divisors $\P^1\backslash X_i$. As in the proof of Lemma \ref{finmultaddmaps} above, we may invoke \cite[Th.\,6.7]{rosmodulispaces} to conclude that it is sufficient to show that, for any reduced divisors $D_1, \dots, D_r$ on $\P^1$, the number of multi-additive maps $\prod_{i=1}^r \R_{D_i/K}(\Gm) \rightarrow U$ is finite and bounded in terms only of $U$ and the degrees of the $D_i$, and this is the content of Lemma \ref{finmultaddmaps}(ii).

Now consider the general case. For technical reasons, we shall find it more convenient to prove a slightly modified version of the proposition. We invoke certain notation from the beginning of \cite[\S7]{rosrigidity}. Namely, we have associated to $f$ a map $\Delta_{f,x_1,\dots,x_r}\colon \prod_{i=1}^r X_i \rightarrow U,$ which enjoys two key properties: (1) $\Delta_{f,x_1,\dots,x_r} = 1$ whenever any coordinate is set to $x_i$, and (2) if $f$ vanishes whenever any coordinate is set to $x_i$, then $\Delta_{f,x_1,\dots,x_r} = f$. Further, the construction of $\Delta_{f,x_1,\dots,x_r}$ is functorial in the obvious sense.

Choose dominant maps $g_i\colon Y_i \rightarrow X_i$ with $Y_i \subset \P^{m_i}$ nonempty open subschemes, and let $d_i$ denote the degree of the complementary divisor of $Y_i$. Let $C$ be such that there are at most $C$ maps as in the proposition whenever all of the $X_i$ are open subsets of $\P^1$ with complementary divisors whose degrees sum to $\sum_i d_i$. Given a set $S$ of $> C$ maps $f\colon \prod_i X_i \rightarrow U$ such that $f$ vanishes whenever some coordinate is set to $x_i$, we will show that two of the $f \in S$ must coincide. For each $f$, let $g_f := f\circ \prod_i g_i\colon \prod Y_i \rightarrow U$. For each $r$-tuple of points $y_i \in Y_i(K)$, consider the map $\Delta_{g_f,y_1,\dots,y_r}\colon \prod Y_i \rightarrow U$, which vanishes whenever some coordinate is set to $y_i$. For each $i$, let $L_i$ denote the space of lines on $\P^{m_i}$ through $y_i$. We may assume that $K$ is infinite (else $r = 0$ and the proposition is trivial). Then for a Zariski dense set of points $\ell \in (\prod L_i)(K)$, one has that $Y_{\ell} := \ell \cap \prod Y_i$ is a product of open subschemes of $\P^1$ with complementary divisor on the $i$th coordinate of degree $d_i$. It follows from our choice of $C$ that two of the $\Delta_{g_f,y_1,\dots,y_r}$ must coincide when restricted to $Y_{\ell}$. Because the $Y_{\ell}$ are Zariski dense in $\prod Y_i$, we must therefore have that two of the $\Delta_{g_f,y_1,\dots,y_r}$ coincide.

Thus we have shown that for each $y \in \prod Y_i(K)$, there are two $f \in S$ such that $\Delta_{g_f,y}$ coincide. It follows that there exist $f_1, f_2 \in S$ such that $\Delta_{g_{f_1},y} = \Delta_{g_{f_2},y}$ for a Zariski dense set of $y \in \prod Y_i(K)$. Because $\Delta_{g_f,y_1,\dots,y_r} = \Delta_{f,g(y_1),\dots,g(y_r)}\circ g$, and $g$ is dominant, it follows that $\Delta_{f_1,f_1(y_1),\dots,f_1(y_r)} = \Delta_{f_2,f_2(y_1),\dots,f_2(y_r)}$ for a Zariski dense set of $\prod y_i \in \prod Y_i(K)$. Thus one has that $\Delta_{f_1,x} = \Delta_{f_2,x}$ for a Zariski dense set of $x \in \prod X_i(K)$, hence for all $x \in \prod X_i(K)$. In particular, this holds for $x = (x_1, \dots, x_r)$. Because, for all $f \in S$, $f$ vanishes whenever some coordinate is set to $x_i$, it follows that $f_1 = f_2$. This proves the proposition.
\end{proof}

\begin{proposition}
\label{imageprodunirpmwd}
Let $K$ be a field of degree of imperfection $r < \infty$, let $(X_1, x_1), \dots, (X_r, x_r)$ be pointed unirational $K$-schemes, and let $U$ be a wound unipotent $K$-group. Finally, let $f\colon \prod_{i=1}^r X_i \rightarrow U$ be such that $f$ vanishes whenever any one of the coordinates is set to $x_i$. Then $f$ generates a permawound $K$-subgroup of $U$.
\end{proposition}

\begin{proof}
We may assume that $K$ is separably closed. We may additionally replace $U$ by the subgroup generated by $f$ and thereby assume that $f$ generates $U$. If $U$ is not permawound, then by Proposition \ref{permawdfinite} there is a wound unipotent $K$-group $V$ such that $\Hom(U, V)$ is infinite. But by postcomposing $f$ with these homomorphisms, we would thereby obtain a violation of Proposition \ref{finmapsprodr}.
\end{proof}

We are now prepared to prove the main result of this section. Let $G$ be a smooth connected $K$-group scheme. Recall that the lower central series $\mathscr{D}_n(G)$ of $G$ is defined by the formula $\mathscr{D}_1(G) := G$, $\mathscr{D}_{n+1}(G) := [G, \mathscr{D}_n(G)]$ for $n \geq 1$. This group must eventually stabilize, and we by define $\mathscr{D}_{\infty}(G)$ to be the stabilization.

\begin{theorem}
\label{unifrmquot}
Let $K$ be a field of degree of imperfection $r > 1$. Then a smooth connected unipotent $K$-group $U$ is unirational if and only if $U/\mathscr{D}_rU$ is.
\end{theorem}

For the corresponding result when $r = 1$, see Corollary \ref{uniriffabis}.

\begin{proof}
Because $\mathscr{D}_{\infty}(U) = 1$, we may assume that $r < \infty$. The only if direction is clear, so we prove the converse. If $U_w$ denotes the maximal wound quotient of $U$ (that is, the quotient of $U$ by its maximal split $K$-subgroup), then $U$ is unirational if and only if $U_w$ is, so we may assume that $U$ is wound. We may of course also assume that $U \neq 1$. Thus $U$ contains a nontrivial smooth connected central $K$-subgroup $U'$. Consider the $r$-fold commutator map $U^r \rightarrow U$, $(u_1, \dots, u_r) \mapsto [u_1, [u_2, \dots, u_r]\dots]$. This descends to a map $c\colon\overline{U}^r \rightarrow U$, where $\overline{U} := U/U'$. (Here we use $r > 1$.) Because $U/\mathscr{D}_rU \twoheadrightarrow \overline{U}/\mathscr{D}_r(\overline{U})$, $\overline{U}$ is unirational by dimension induction. If $c$ is constant, then $\mathscr{D}_r(U) = 1$ and the theorem is trivially true. Otherwise, Proposition \ref{imageprodunirpmwd} implies that $c$ generates a nontrivial (normal) permawound $K$-subgroup $V$ of $U$. We then finish by combining dimension induction applied to $U/V$ and Proposition \ref{extunirgpspmwdunir}.
\end{proof}

\begin{remark}
If $U$ is a wound unirational $K$-group, and $K$ has degree of imperfection $r$, then $\mathscr{D}_{r+1}(U) = 1$ \cite[Th.\,1.5]{rosrigidity}. (Note that the indexing on the lower central series used in that paper is unfortunately nonstandard, and in particular differs from that used here by $1$.) Therefore the quotient $U/\mathscr{D}_r(U)$ is in that sense not too far from $U$.
\end{remark}

Of particular note in Theorem \ref{unifrmquot} is the case $r = 2$, which says that if $K$ has degree of imperfection $2$, then $U$ is unirational if $U^{\ab}$ is. The same holds for $r \leq 2$, by Corollary \ref{uniriffabis} and because every connected linear algebraic group is unirational when $r = 0$. We shall now give an example to show that this is optimal, by constructing over every field of degree of imperfection $\geq 3$ a wound unipotent group $U$ such that $U^{\ab}$ is $p$-torsion and unirational even though $U$ fails to be unirational.

\begin{example}
\label{Uabunirex}
Let $K$ be a field of degree of imperfection $\geq 3$, and let $\lambda, \mu, \gamma \in K$ be $p$-independent. Recall that we have defined the following $K$-group
\[
V_{1,\gamma} := \left\{-X_{p-1} + \sum_{i=0}^{p-1}\gamma^iX_i^p = 0\right\}
\]
and $V_{1,\lambda}$, $V_{1,\mu}$ similarly. We also define the group
\[
W := \left\{-Z_{p-1,p-1} + \sum_{0 \leq i,j < p} \lambda^i\mu^jZ_{i,j}^p = 0\right\},
\]
where the variables in the $W$ equation are $Z_{i,j}$ for $0 \leq i,j < p$. We have a bi-additive map $$b\colon V_{1,\lambda} \times V_{1,\mu} \rightarrow W$$ defined by $(X_i)_i \times (Y_j)_j \mapsto (Z_{i,j} := X_iY_j)_{i,j}$. That this does indeed land in $W$ is seen by the following calculation: $$Z_{p-1,p-1} = X_{p-1}Y_{p-1} = \left(\sum_{i=0}^{p-1}\lambda^iX_i^p\right)\left(\sum_{j=0}^{p-1}\mu^jY_j^p\right) = \sum_{0\leq i,j < p} \lambda^i\mu^j(X_iY_j)^p = \sum_{0\leq i,j < p} \lambda^i\mu^jZ_{i,j}^p.$$
We claim that $b$ generates $W$. That is, $b$ does not land in any proper $K$-subgroup scheme of $W$, or equivalently, for some $n > 0$ the map $b^n\colon (V_{1,\lambda} \times V_{1,\mu})^n \rightarrow W$ obtained by adding $b$ over the $n$ coordinates is surjective. In fact, since all groups and maps live over the subfield $\F_p(\lambda, \mu)$, it suffices to verify this over this subfield, where it is \cite[Prop.\,9.9]{rosrigidity}. (The multi-additive map given in the proof of that proposition is in this setting exactly $b$.)

Let $$G := V_{1,\lambda} \times V_{1,\mu} \times V_{1,\gamma},$$ and let $X := W \times G$. For any bi-additive map $h\colon G\times G \rightarrow W$, we obtain a group structure $X_h$ on $X$ via $(w_1,g_1)\cdot(w_2,g_2) := (w_1+w_2+h(g_1,g_2), g_1+g_2)$. The projection map $X_h \rightarrow G$ is then a homomorphism with kernel (isomorphic to) $W$, and this extension of $G$ by $W$ is central. Furthermore, one readily verifies that the commutator of $(0, g_1), (0, g_2)$ is $h(g_1,g_2)-h(g_2,g_1) \in W$. We apply these observations to the bi-additive map $h_0\colon G^2 = (V_{1,\lambda} \times V_{1,\mu} \times V_{1,\gamma})^2 \rightarrow W$ defined by $h((v_1,v_2,v_3), (v_1',v_2',v_3')) := b(v_1,v_2')$. Let $U_1 := X_{h_0}$. If we write $g_i := (x_i,y_i,z_i)$, then $[(0,g_1),(0,g_2)] = b(x_1,y_2) - b(x_2,y_1)$. In particular, setting $x_2 = 0$ and using the fact that $b$ generates $W$, we see that $U_1^{\ab} = G$.

Now we define another extension $U_2$ of $G$ by $W$ as follows. By definition, the group $W$ sits in an exact sequence 
\begin{equation}
\label{Uabunirexeqn1}
0 \longrightarrow W \longrightarrow \Ga^{p^2} \xlongrightarrow{F} \Ga \longrightarrow 0,
\end{equation}
where $F((Z_{i,j})_{ij}) := -Z_{p-1,p-1} + \sum_{0 \leq i,j < p} \lambda^i\mu^jZ_{i,j}^p.$ Associated to (\ref{Uabunirexeqn1}) we obtain a connecting map $\delta\colon \Hom(G, \Ga) \rightarrow \Ext^1(G, W)$. Let $\chi\colon G \rightarrow \Ga$ denote the homomorphism $G \rightarrow V_{1,\gamma} \xrightarrow{X_{p-1}} \Ga$, where the first map is projection. Then we define $U_2 := \delta(\chi)$. We claim that $U_2$ is not unirational. Let us grant this claim for the moment, and use it to construct our example of a unipotent $K$-group $U$ such that $U$ is not unirational but $U^{\ab}$ is.

We define $U$ to be the Baer sum (which makes sense -- and defines an abelian group structure -- on the set of isomorphism classes of {\em central} extensions) of the two extensions $U_1$ and $U_2$ of $G$ by $W$. That is, $U := \coker(\Delta\colon W \rightarrow U_1 \times_G U_2)$, where $\Delta$ is the antidiagonal map $w \mapsto (w, -w)$. We first check that $U$ is not unirational. Indeed, the forgetful map $\Ext^1(G, W) \rightarrow {\rm{H}}^1(G, W)$ which recalls only the structure of the extension as a $W$-torsor over $G$ is a homomorphism. Since $U_1$ is by definition trivial as a torsor, it follows that $U \simeq U_2$ as a $W$-torsor over $G$. In particular, they are isomorphic as $K$-schemes. Since the latter is not unirational, therefore, neither is the former. Now a straightforward calculation (most easily carried out by regarding the derived group as an fppf sheaf rather than as an algebraic group) using the fact that $W \subset \mathscr{D}U_1$ shows that $W \subset \mathscr{D}U$, hence $U^{\ab} = G$ is unirational.

It remains to prove that $U_2$ is not unirational. The argument is more or less the same as in Example \ref{extwdbywdnotunirex}. As in that example, it is sufficient to show that there is no nonzero homomorphism (with $n > 0$ and $\alpha \in K - K^p$) $f\colon V_{n,\alpha} \rightarrow V_{1,\gamma}$ such that $X_{p-1} \circ f$ lifts through the map $F$ in (\ref{Uabunirexeqn1}) to a homomorphism $V_{n,\alpha} \rightarrow \Ga^{p^2}$. (See (\ref{eqnforVnalpha}) for the definition of $V_{n,\alpha}$.)
Suppose to the contrary that there were. Then by Lemma \ref{mapVnalphtoVaalph}, we would have $\alpha \in K^p(\gamma)$, and using $S, S_k$ to denote the variables on $V_{n,\alpha}$, we may write 
\begin{equation}
\label{Uabunirexeqn4}
X_{p-1}\circ f = cS, \hspace{.2 in} 0 \neq c \in K.
\end{equation}
Then we have an equation
\begin{equation}
\label{Uabunirexeqn2}
cS = -Z_{p-1,p-1} + \sum_{0 \leq i,j < p} \lambda^i\mu^jZ_{i,j}^p,
\end{equation}
with each $Z_{i,j}$ a $p$-polynomial in $S$ and the $S_k$ of degree $\leq 1$ in $S$. We claim that $\mathrm{deg}(Z_{i,j}) \leq p^{n-1}$. Indeed, using the equation (\ref{eqnforVnalpha}) to eliminate $S^p$, (\ref{Uabunirexeqn2}) becomes an identity of polynomials. If some $Z_{i,j}$ has degree $> p^{n-1}$, then looking at the leading term in some $S_k$ of (\ref{Uabunirexeqn2}) and using the $p$-independence of $\lambda, \mu$ would yield a contradiction. Thus we may write 
\begin{equation}
\label{Uabunirexeqn3}
Z_{i,j} = d_{i,j}S + \sum_{\substack{0 \leq k < p^n\\ k \not\equiv -1\pmod{p}}} d_{i,j,k}S_k^{p^{n-1}} + G_{i,j},
\end{equation}
where $d_{i,j}, d_{i,j,k} \in K$ and $G_{i,j}$ is a $p$-polynomial over $K$ in the $S_k$ (not involving $S$) of degree $< p^{n-1}$. Then using (\ref{eqnforVnalpha}) to eliminate $S^p$, and substituting (\ref{Uabunirexeqn3}) into (\ref{Uabunirexeqn2}) and comparing coefficients of $S_0^{p^n}$ yields $$\alpha^{1-p}\sum_{0 \leq i,j < p} \lambda^i\mu^jd_{i,j}^p = \sum_{0 \leq i,j < p} \lambda^i\mu^jd_{i,j,0}^p.$$ If not all $d_{i,j} = 0$, then it would follow that $\alpha \in K^p(\lambda, \mu)$. But we have that $\alpha \in K^p(\gamma)$, so it would then follow (because $\lambda, \mu, \gamma$ are $p$-independent) that $\alpha \in K^p$, which is false. Thus $d_{i,j} = 0$ for all $i,j$. Combining (\ref{Uabunirexeqn3}) and (\ref{Uabunirexeqn2}) then yields that $c = 0$, in violation of (\ref{Uabunirexeqn4}). This completes the proof that $U_2$ is not unirational.
\end{example}

\end{document}